\documentclass[12pt]{amsart}

\newtheorem{theorem}{Theorem}[section]

\newtheorem{corollary}[theorem]{Corollary}

\theoremstyle{definition}

\newtheorem{definition}[theorem]{Definition}

\newtheorem{remark}[theorem]{Remark}

\theoremstyle{parrafo}

\begin{document}

\title[]{Boundedness of averaging operators on geometrically doubling metric spaces}

\author{J. M. Aldaz}
\address{Instituto de Ciencias Matem\'aticas (CSIC-UAM-UC3M-UCM) and Departamento de 
Matem\'aticas,
Universidad  Aut\'onoma de Madrid, Cantoblanco 28049, Madrid, Spain.}
\email{jesus.munarriz@uam.es}
\email{jesus.munarriz@icmat.es}

\thanks{2010 {\em Mathematical Subject Classification.} 41A35}

\thanks{The author was partially supported by Grant MTM2015-65792-P of the
MINECO of Spain, and also by by ICMAT Severo Ochoa project SEV-2015-0554 (MINECO)}

%\thanks{Supported by DGICYT 1317 DP Spain}

%\subjclass{}

%\keywords{}

%\date{}

%\dedicatory{}

%\commby{} 

%%% ----------------------------------------------------------------------
\begin{abstract} We prove that averaging operators are uniformly bounded on $L^p$ for all 
 geometrically doubling metric measure spaces and all $1 \le p < \infty$, with bounds independent
of the measure.   From this result, the $L^p$ convergence of averages as $r \to 0$
immediately follows.
\end{abstract}

%%% ----------------------------------------------------------------------

\maketitle

%%% ----------------------------------------------------------------------

\markboth{J. M. Aldaz}{Averaging operators}

\section {Introduction} 

It is a  well known  consequence of translation invariance that  for Lebesgue measure
on $\mathbb{R}^d$, the averages $A_r f$ (also known as Steklov means) 
converge to
$f$ in $L^1$ as $r\to 0$. Rather surprisingly, the corresponding approximation question 
regarding arbitrary locally finite
Borel measures $\mu$ on $\mathbb{R}^d$ (i.e.,  whether 
$\lim_{r\to 0}A_{r, \mu} f = f$ in $L^1(\mu)$)
does not seem to have been
studied in the literature.

 Here we give an affirmative answer not just for  $\mathbb{R}^d$, but  for every geometrically doubling
 metric measure space (every metric space of homogeneous type, in the terminology of \cite{CoWe}, 
cf. Definition  \ref{geomdoub} below).  We do this by proving that 
averaging operators are uniformly bounded on $L^1$,  something that was previously unknown even for  $\mathbb{R}^d$.  
By interpolation, boundedness also holds for all $p\in (1,\infty)$, since the case
$p = \infty$ is trivial. Note that on $\mathbb{R}^d$ and for
$p\in (1,\infty)$, the
$L^p$ boundedness of averaging operators was already known, since
it is immediate from the correspondig boundedness of the centered
Hardy-Littlewood maximal operator. But for  geometrically doubling
metric spaces in which the Besicovitch covering theorem does not hold,
not only the case $p=1$, but also the case $p\in (1,\infty)$, is new.

The boundedness question 
regarding euclidean spaces was asked
by the author in \cite{Al1}; 
 two special cases were proved in that paper: the exponential distribution in
 one dimension, and the standard gaussian distribution  in every
 dimension. Here we present the general result.

\section {Definitions and notation} 

We will use $B^{o}(x,r) := \{y\in X: d(x,y) < r\}$ to denote open balls, 
and 
$B^{cl}(x,r) := \{y\in X: d(x,y) \le r\}$ to refer to metrically closed balls (``closed ball" will always be understood in the metric, not the
topological sense). 
 If we do not want to specify whether balls are open or closed,
we write $B(x,r)$. But when we utilize $B(x,r)$, we assume that all balls are of the same kind, i.e., all open or all closed.

\begin{definition} A Borel measure is   {\em $\tau$-additive} or {\em $\tau$-smooth}, if for every
collection  $\{U_\alpha : \alpha \in \Lambda\}$
 of  open sets, 
$$
\mu (\cup_\alpha U_\alpha) = \sup_{\mathcal{F}} \mu(\cup_{i=1}^nU_{\alpha_i}),
$$
 where the supremum is taken over all finite subcollections $\mathcal{F} = \{U_{\alpha_1}, \dots, U_{\alpha_n} \}$
of  $\{U_\alpha : \alpha \in \Lambda\}$.
 We say that $(X, d, \mu)$ is a {\em metric measure space} if
$\mu$ is a  $\tau$-additive  Borel measure on the metric space $(X, d)$, such that $\mu$ assigns finite measure
to bounded Borel sets. 
\end{definition} 

From now on we always assume that measures are locally finite (finite on bounded sets)
and not
identically 0.
For motivation regarding the definition of metric measure spaces using $\tau$-additivity,
cf. \cite{Al3}. Note that
in separable metric spaces all Borel measures are  $\tau$-additive (so the spaces considered above are more general than 
those given by some commonly used alternative definitions, cf. \cite{HKST} for instance)
 and the same happens
with  all Radon measures in arbitrary metric spaces.

Recall that 
the complement of
the support $(\operatorname{supp}\mu)^c := \cup \{ B^{o}(x, r): x \in X, \mu B^{o}(x,r) = 0\}$
of a Borel  measure
 $\mu$,  is an open set, and hence measurable. 

\begin{definition}\label{maxfun} Let $(X, d)$ be a metric space and let
$\mu$ be a locally finite Borel measure on $X$. 
If $\mu (X \setminus \operatorname{supp}\mu) = 0$, 
we say that $\mu$ has {\em full support}. 
\end{definition}

By $\tau$-additivity,  if $(X, d, \mu)$ is  a metric measure space, then 
$\mu$ has full support,
since $X \setminus \operatorname{supp}\mu $ is a union of open balls of measure zero.
Actually, the other implication also holds, for the support is always separable,  so  having full support is equivalent to
$\tau$-additivity (cf. \cite[Proposition 7. 2. 10]{Bo} for more details).

\begin{definition}\label{maxfun} Let $(X, d, \mu)$ be a metric measure space and let $g$ be  a locally integrable function 
on $X$. For each fixed $r > 0$ and each $x\in \operatorname{supp}\mu$, the
averaging operator $A_{r, \mu}$ is defined as
\begin{equation}\label{avop}
A_{r , \mu} g(x) := \frac{1}{\mu
(B(x, r))} \int _{B(x, r)}  g \ d\mu.
\end{equation}
\end{definition}

 Averaging operators in metric measure spaces are defined almost everywhere,   by
$\tau$-additivity. Sometimes it is convenient to  specify whether balls are open or closed; in that case,
 we use $A_{r , \mu}^{o} $ and $A_{r , \mu}^{cl} $ for the corresponding operators. 
 Furthermore, when we are considering only one measure $\mu$ we often omit it, 
writing $A_{r } $ instead of the longer $A_{r , \mu} $. 

\vskip .3cm

\begin{definition} Let $(X, d)$ be a metric space. A {\em strict $r$-net} (resp. {\em non-strict $r$-net}) 
 in $X$ is a subset $S \subset X$ such that for any pair of distinct points $x,y \in S$, we have $d(x,y) > r$  (resp. $d(x,y)  \ge r$). 
\end{definition} 

We  speak of an {\em $r$-net}   if we do not want to specify whether it is strict or not.
To ensure disjointness of the balls $B(x,r/2)$,  $r$-nets are always  taken to be strict when working with closed balls;
otherwise, we assume $r$-nets are
non-strict.

\begin{definition} \label{geomdoub} A metric space is {\it geometrically doubling}  if there exists a positive
integer $D$ such that every ball of radius $r$ can be covered with no more than $D$ balls
of radius $r/2$.  We call the smallest such $D$ the {\em doubling constant} of the space.
\end{definition}

We use $D^{o}$ and $D^{cl}$  to refer to the corresponding constants for open and for closed balls.
It is easy to see, by enlarging balls slightly, that the geometrically doubling condition is satisfied for open balls if and only if
it is satisfied for closed balls. But the constants will in general be different; for instance, if $X= \mathbb{R}$, 
then $D^{o} = 3$ and $D^{cl} = 2$. 

\begin{remark}
Let $X$ be geometrically doubling with constant $D$, and let $M$ be the maximum size of an $r$-net in $B(x, r)$,  taken over 
all $x\in X$ and all  $r > 0$. Then $M \le D$, since every point in a maximal $r$-net
inside $B(x,r)$ is contained in one of the covering balls of radius $r/2$, and each
such ball can contain at most one point from the $r$-net. By analogy with previous notation, we use  $M^{cl}$   and $M^{o}$ for strict and non-strict nets
respectively.\end{remark}

\section{Boundedness of averaging operators on geometrically doubling spaces} 

The following proof reminds the reader of the fact that bounded continuous functions with
bounded support are dense in $L^p$ for $1 \le p < \infty$, and for these functions, averages converge in norm as
$r\to 0$.

\begin{theorem} \label{L1}{\bf $L^p$-Lebesgue differentiation}.   \label{avimplyL1} Let $(X, d, \mu)$ be a
metric measure space, and let $1 \le p < \infty$.
If there is a constant $C > 0$ such that 
$\sup_{r > 0} \|A_r \|_{L^p(\mu)\to L^p(\mu)} \le C,
$ 
then  for every $f\in L^p(\mu)$, $\lim_{r\to 0} A_{r} f =  f$ in $L^p$. 
 \end{theorem} 
 
 \begin{proof}  Given any sequence $\{r_n\}_{n\ge 1}$ satisfying $r_n\to 0$,
 we may suppose that $r_n \le 1$, by disregarding a finite number of terms, if needed.
We may also suppose that $C\ge1$ (else, replace it by 1).

Recall that in metric measure
spaces the  continuous functions that belong to $L^p$ are dense in $L^p$, by the standard argument
whereby the case of  real valued functions   is reduced to the case of non-negative functions,
which by successive approximations is reduced first, to the case of simple functions, then
to indicator functions of measurable sets, and finally, to indicator functions of closed sets $F$
(with finite measure); for these functions the result is true by the Tietze-Urysohn extension theorem:
given $\varepsilon > 0$ we choose $O$ open such that $F\subset O$ and $\mu (O\setminus F)
 < \varepsilon$; then we extend $g : O^c \cup F \to \{0,1\}$ given by $g = 1$ on $F$, $g = 0$ on $O^c$,
 to a continuous function  $G : X \to [0,1]$.

 Let 
 $z\in X$, let $\varepsilon > 0$, and let $f\in L^p(\mu)$. Then there exist 
 $t \gg 0$ and $ R \gg 0$ such that
 $$
 \|f - f  \ \mathbf{1}_{B(z, R) \cap \{|f| \le t\}} \|_p <  
  \frac{\varepsilon }{6 C}.
 $$ 
Next we choose a continuous function $g$ such that $- t \le g \le t$, 
$\operatorname{supp} g \subset B(z, R+ 1)$ and 
 $$
  \| f  \ \mathbf{1}_{B(z, R) \cap \{|f| \le t\}} - g \|_p 
 < \frac{\varepsilon }{6 C}.
 $$ 
 By the triangle inequality,
 $$
  \| f  - g \|_p 
 < \frac{\varepsilon }{3 C}.
 $$
 Now $g$ is continuous, so for all $x\in X$, 
 $\lim_n A_{r_n} g(x)   =  g(x)$. Since for all $n\ge 1$, 
 $$
 |A_{r_n} g(x)   -  g(x)|^p
 \le (2 t)^p  \ \mathbf{1}_{B(z, R + 2)}  (x) \in L^1(\mu),
 $$
  by the dominated convergence
 theorem, 
 $\lim_n \|A_{r_n} g  - g \|_p = 0$. Hence, 
  $$
 \limsup_n \|A_{r_n} f  - f\|_p 
 \le 
 \sup_n \|A_{r_n} f  - A_{r_n} g \|_p + \lim_n \|A_{r_n} g  - g \|_p +  \|g  - f \|_p 
 < \frac{\varepsilon}{3} +  \frac{\varepsilon}{3 C} 
 < \varepsilon.
 $$
\end{proof}

\begin{definition} \label{conjugate} We call 
\begin{equation}\label{conjug}
a_s (y)  
: =
 \int_X  \frac{\mathbf{1}_{B(y,s)}(x)}{\mu B(x,s)}  \  d\mu(x) 
 \end{equation}
the {\em conjugate function} to the averaging operator $A_s$.
\end{definition}

As  it happens with averaging operators, the conjugate function $a_s$ is well defined a.e., when $y$ belongs to
 the support of $\mu$. If one wishes,  it can be defined
everywhere via the usual conventions $0/0 = 0 \cdot \infty = 0$, and $1/0 = \infty$. These handle all the cases
where the denominator vanishes.

\begin{theorem}  \label{equiv} Let $(X, d, \mu)$ be a metric measure space.
 The averaging operator $A_s$ is bounded on
$L^1(\mu)$ if and only if $a_s \in L^\infty (\mu)$, in which case 
$\|A_s \|_{L^1(\mu)\to L^1(\mu)} = \|a_s\|_\infty. 
$
 \end{theorem} 

\begin{proof}  Let $0 \le f \in L^1(\mu)$, and suppose  $a_s \in L^\infty (\mu)$. Since by
Fubini-Tonelli
\begin{equation}\label{fubini}
\|A_s f\|_{L^1} 
=
\int_X A_s f(x)  \  d\mu(x)
=
\int_X\int_X  \frac{\mathbf{1}_{B(x,s)}(y) }{\mu B(x,s)} f (y) \  d\mu(y) \  d\mu(x)
\end{equation}
\begin{equation}\label{fubini2}
= 
\int_X  f (y) \int_X  \frac{\mathbf{1}_{B(y,s)}(x)}{\mu B(x,s)}  \  d\mu(x) \  d\mu(y)
= 
\int_X  f (y) \ a_s(y) \  d\mu(y),
\end{equation}
it follows from H\"older's inequality that $\|A_s \|_{L^1(\mu)\to L^1(\mu)} \le \|a_s\|_\infty. 
$

On the other hand, we claim that if $\|A_s \|_{L^1(\mu)\to L^1(\mu)} \le C$, then $\|a_s\|_\infty
\le C. 
$
Towards a contradiction, suppose $C < \|a_s\|_\infty$ (including the case $\|a_s\|_\infty
= \infty$). Then there is a $t > C$ and a measurable set $A_t$ such that 
$A_t \subset \{ a_s > t\}$ and $0 < \mu A_t  < \infty$. Let $f:= \mathbf{1}_{A_t} \in L^1(\mu)$.
Then 
\begin{equation}\label{fubini3}
\|A_s f\|_{L^1} 
=
\int_X  f (y) \  a_s(y) \  d\mu(y)
>  \int_{A_t}  t  \  d\mu(y)
= t \mu (A_t) > C \|f\|_{L^1}.
\end{equation}
\end{proof}

\begin{definition} \label{loccomp}  We say that a measure $\mu$ satisfies a {\it local comparability condition} 
if there
exists a constant $C\in[1, \infty)$ such that for  all pairs of points $x,y\in X$
and every $r > 0$,
whenever $d(x,y) < r$, we have 
$\mu(B(x,r))\le C \mu(B(y,r)).$
\end{definition}

The preceding definition comes from \cite[p.  737]{NaTa}. There,  local comparability  is called a ``mild uniformity assumption";  the term ``local comparability"
was introduced in \cite{Al1}, and used also in \cite{Al2}.
As indicated in \cite[p. 737]{NaTa}, if $\mu$ satisfies a  $C$ local comparability condition, then
$
a_s (y)  \le C,
$
so $\|A_s\|_{L^1\to L^1} \le C$.

It is natural to ask under which conditions one can have  uniform boundedness of $A_r$  without
 local comparability.  In  \cite[Example 4.1]{Al1}   a metric measure space
 is exhibited
where the measure lacks local comparability,
and $\|A_s\|_{L^p\to L^p}$ is unbounded for all $p\in [1, \infty)$.
Also, it is shown in  \cite[Theorem  4.8]{Al1}   that on $\mathbb{R}$,  there is a measure $\mu$ such that  
 the right directional averaging operator
$$
A_{s,\mu}^r f(x) := \frac{1}{\mu ([x, x + s])} 
\int_{[x, x + s]}  f (y)   \  d \mu (y)
$$
is unbounded on $L^1(\mu)$ for $s = 1$.

On the other hand, the
uniform $L^1$  boundedness of the operators $A_r$ was shown to hold for $\mathbb{R}^d$ in two special cases, 
despite the lack of local comparability: the
exponential density in dimension one, and the standard gaussian measures in every dimension (cf.  \cite[Theorems  4.2 and 4.3]{Al1}).

Here we obtain the general result: Averaging operators are 
 $L^1$  bounded, uniformly on $r$,  for every 
locally finite measure in any geometrically doubling metric measure
space, 
not just in $\mathbb{R}^d$.

\begin{theorem}  \label{geomdoubling} Let $(X, d, \mu)$ be a geometrically doubling 
metric measure space, with doubling constant $D$, 
and let
$M$ be the maximum 
cardinality of any $r$-net in $B(x, r)$, where the maximum is taken
over all  $x \in X$ and all $r >0$. 
Then for every $s > 0$, 
$\|A_s \|_{L^1(\mu)\to L^1(\mu)} \le M \le D.
$ 
Since $\|A_s \|_{L^\infty(\mu)\to L^\infty(\mu)} = 1, 
$ for $1 < p < \infty$ we have $\|A_s \|_{L^p(\mu)\to L^p(\mu)}  \le M^{1/p}
$
by interpolation. 
 \end{theorem} 

\begin{proof} 
Let $X$ be geometrically doubling with constant $D$, and let $M$ be the maximum size of an $r$-net in $B(x, r)$. As was noted before, $M \le D$.
By disregarding a set of measure zero if needed, we suppose that 
$X = \operatorname{supp} \mu$, so every ball has positive measure.
Fix $y\in X$. We want to show that $a_s(y) \le M$, and then the result follows  from
Theorem \ref{equiv}. 

First we claim that $b_1 := \inf \{\mu B(x,s) : x \in B(y,s)\} > 0$. To see why, select
a sequence $\{x_n\}_{n\ge 1}$ of points in $B(y,s)$ so that
$\lim_n \mu B(x_n, s) = b_1$. Since $X$ is geometrically doubling, $B(y,s)$  
can be covered by at most $D$ balls of radius $s/2$, so at least one of these
balls, say, $B(w, s/2),$ contains an infinite  subsequence from $\{x_n\}_{n\ge 1}$, 
which after relabelling, we also denote by $\{x_n\}_{n\ge 1}$. 
Then $b_1 \ge \mu B(w, s/2) > 0$, since for all $n\ge 1$, $B(w, s/2) \subset B(x_n, s)$.
Now take $0 < \varepsilon \ll 1$, and choose $u_1\in B(y,s)$ so that
$\mu B(u_1, s) < (1 + \varepsilon) b_1$; let 
$b_2 := \inf \{\mu B(x,s) : x \in B(y,s) \setminus B(u_1, s) \}$, and select
$u_2\in B(y,s) \setminus B(u_1, s) $ so that $\mu B(u_2, s) < (1 + \varepsilon) b_2$;
repeat, with 
$b_{k + 1} := \inf \{\mu B(x,s) : x \in B(y,s) \setminus \cup_1^k B(u_i, s)  \}$,
$u_{k + 1}\in B(y,s) \setminus \cup_1^k B(u_i, s) $, and 
$\mu B(u_{k + 1}, s) < (1 + \varepsilon) b_{k + 1}$. Since the points $u_i$ form an $s$-net
in $B(y,s)$,
there is an $m\le M$ such that
$B(y,s) \setminus \cup_1^m B(u_i, s) = \emptyset$, and then the process stops.

Next, fix $x\in B(y,s)$, and let $i$ be the first index such that $x\in B(u_i,s)$. Then
$$
 \frac{\mathbf{1}_{B(y,s)}(x)}{\mu B(x,s)} 
 \le 
 (1 + \varepsilon) \frac{\mathbf{1}_{B(y,s) \cap B(u_i ,s)}(x)}{\mu B(u_i ,s)} 
 \le 
 (1 + \varepsilon)  \sum_{j=1}^m\frac{\mathbf{1}_{B(y,s) \cap B(u_j ,s)}(x)}{\mu B(u_j ,s)},
$$
so
\begin{equation*}
a_s (y)  
=
 \int_X  \frac{\mathbf{1}_{B(y,s)}(x)}{\mu B(x,s)}  \  d\mu(x) 
 \le
  \int_X  (1 + \varepsilon)  \sum_{j=1}^m\frac{\mathbf{1}_{B(y,s) \cap B(u_j ,s)}(x)}{\mu B(u_j ,s)}  \  d\mu(x) 
 \end{equation*}
\begin{equation*}
\le 
 (1 + \varepsilon)   \int_X \sum_{j=1}^m\frac{\mathbf{1}_{B(u_j ,s)}(x)}{\mu B(u_j ,s)}  \  d\mu(x) 
\le
  (1 + \varepsilon)  M,
 \end{equation*}
and 
$\|A_s \|_{L^1(\mu)\to L^1(\mu)} \le M 
$ follows by letting $\varepsilon \downarrow 0$.

It is obvious that for all $s > 0$, $\|A_s \|_{L^\infty(\mu)\to L^\infty(\mu)} = 1$
(we always have $\|A_s \|_{L^\infty(\mu)\to L^\infty(\mu)} \le  1$, since averages never exceed a 
supremum; for the other inequality,  just take $f = \mathbf{1}_X$). Thus, by a standard interpolation argument,
or simply by Jensen's inequality (cf. \cite[Theorem 2.10]{Al1})
for all  $1 < p < \infty$ we have $\|A_s \|_{L^p(\mu)\to L^p(\mu)}  \le M^{1/p}$.
\end{proof}

While the bounds given in the preceding theorem do not seem very tight, they are.
Adapting some arguments  from \cite{Al4}, we show next that it is possible to have 
$\|A^{cl}_{1,  \mu} \|_{L^1(\mu)\to L^1(\mu)} =  M^{cl} = D^{cl},
$ 
so the bounds from  Theorem  \ref{geomdoubling}  cannot in general
be improved. 
The exponential 
dependency on the dimension of the space was already known in some natural cases: 
 for the 
standard gaussian measure  $\gamma^d$  on  $(\mathbb{R}^d, \|\cdot\|_2)$, for $1 \le p < \infty$, and
for  every $d$ sufficiently large, the weak type $(p,p)$
constants satisfy $\left\|A_{\frac{\sqrt{3d - 3}}{2}} \right\|_{L^p\to L^{p, \infty}} >1.019^{d/p}$, cf. \cite[Theorem 4.3]{Al1}. 

Recall that balls with respect to $\|\cdot\|_\infty$,  the  $\ell_\infty$ norm on 
 $\mathbb{R}^d$, are cubes with sides parallel to the axes.

\begin{theorem} \label{kiss}  There exists a discrete measure $\mu$ on 
 $(\mathbb{R}^d, \|\cdot\|_\infty)$
  such that $ \|A^{cl}_{1,  \mu} \|_{L^1(\mu)\to L^1(\mu)} =  2^d =M^{cl} = D^{cl}$.
\end{theorem}

\begin{proof}  That $2^d =M^{cl} = D^{cl}$ is clear: consider $B^{cl} (0,1) = [-1, 1]^d$; its vertices form a strict
1-net in $B^{cl} (0,1)$, so $2^d \le M^{cl}$. Also, the translates of $[0,1]^d$ that  are contained in 
$ [-1, 1]^d$ and share a vertex with  $[-1, 1]^d$ form a cover of  $[-1, 1]^d$, so $2^d \ge  D^{cl}$.

Let  $\{x_1, \dots, x_{2^d}\}$ be an enumeration of the vertices of   $[-3/4, 3 /4]^d$. 
For $n \ge 1$, set
$\mu_n := n^{-1} \delta_{3 n e_1} + \sum_{i= 1}^{2^d} \delta_{x_i + 3 n e_1}$,
and let $\mu := \sum_{n=1}^\infty \mu_n$. If  $f_n := n \mathbf{1}_{3 n e_1}$, then $\| f_n \|_{L^1(\mu)}  = 1$
and $\|A^{cl}_{1,  \mu} f _n \|_{L^1(\mu)} >  2^d n /(n + 1)$.
\end{proof}

\vskip .2 cm

Putting together Theorems \ref{L1},  \ref{equiv}, and \ref{geomdoubling},  we obtain the following

\begin{corollary}  \label{L1conv}
Suppose  $(X, d, \mu)$ is either a geometrically doubling 
metric measure space,  or $\mu$ satisfies a local comparability condition. 
Then for  every $f\in L^p(\mu)$,  $1 \le p < \infty$,
we have $\lim_{r\to 0}  A_{r} f  =  f$ in $L^p$.
 \end{corollary} 
 
 \begin{corollary}  
Suppose  $(X, d, \mu)$ is either a geometrically doubling 
metric measure space,  or $\mu$ satisfies a local comparability condition. 
 Then
for all $f\in L_{loc}^1(\mu)$,
$Mf (x) \ge |f|(x)$ almost everywhere.
 \end{corollary} 
 
 \begin{proof} This follows from Corollary \ref{L1conv}, since from $L^1$ convergent sequences
one can always extract subsequences converging a.e., and for every $r > 0$, 
$Mf (x) \ge A_r |f|(x)$.
\end{proof}
 
  The part of the preceding corollary dealing with geometrically doubling metric spaces had been originally obtained (cf. \cite[Corollary 2.10]{Al3})
 by using a result of Hyt\"onen (cf. \cite[Lemma 3.3]{Hy}) 
  on the existence of arbitrarily small doubling balls,
  for general measures in geometrically doubling metric measure spaces.

In view of the  exponential increase of the  bounds, for the
 standard gaussian measure  $\gamma^d$  on $(\mathbb{R}^d, \|\cdot\|_2)$, one might suspect that in the  infinite dimensional case
the uniform boundedness of the averaging operators
can fail. It follows from a  result of D. Preiss  
that this is indeed the case.

\begin{corollary}\label{preiss} There is a gaussian measure $\gamma$ on an 
infinite
dimensional  separable Hilbert space $H$,
for which $\sup_{r > 0}\|A_r \|_{L^p(\gamma)\to L^p(\gamma)} = \infty
$ whenever $1 \le p < \infty$.
\end{corollary}

\begin{proof} By \cite{Pr}, there exists a gaussian probability $\gamma$ on a separable Hilbert space $H$
and a Borel subset $C$ with $0 < \gamma C < 1$, such that $\gamma$-a.e. x,   
$$
\lim_{r\to 0} \frac{\gamma (C \cap B (x, r))}{\gamma
(B (x, r))} = 0.
$$
Fix $p\in [1, \infty)$. By Theorem \ref{avimplyL1}, if $\sup_{r > 0}\|A_r \|_{L^p(\mu)\to L^p(\mu)} < \infty$, then
$
\lim_{r\to 0} \|A_r \mathbf{1}_C - \mathbf{1}_C\|_p  = 0,
$
 so it is possible to extract a 
 sequence $r_n \to 0$ such that $\gamma$ a.e. $x$,
$$
\lim_{n} \frac{\gamma (C \cap B (x, r_n))}{\gamma
(B (x, r_n))} = \mathbf{1}_C (x),
$$
which is a contradiction.
\end{proof}

\end{document}